\documentclass{amsart}
\usepackage{geometry}
\usepackage[all]{xy}
\usepackage{xspace}
\usepackage{amssymb,amsmath,mathtools}

\usepackage{hyperref}

\newtheorem{assumption}{Assumption}
\numberwithin{equation}{section}
\theoremstyle{plain}
\newtheorem{thm}[equation]{Theorem}
\theoremstyle{definition}
\newtheorem{definition}[equation]{Definition}
\newtheorem{lem}[equation]{Lemma}

\theoremstyle{definition}
\newtheorem{rem}[equation]{Remark}

\newcommand{\SCAL}{{\cdot}}
\newcommand{\SSCAL}{{:}}
\newcommand{\DIV}{\nabla\!{\cdot}}   
\newcommand{\GRAD}{\nabla}           
\newcommand{\LAP}{{\Delta}}          
\def\Ldeuxd{{{ \bf L}^2   (\Omega)}}
\def\Ldeux{{{  L}^2   (\Omega)}}
\def\Hunzd{{{  \bf H}^1_0 (\Omega)}}
\def\Hunz{{{   H}^1_0 (\Omega)}}

\def\Hun{{{    H}^{1}(\Omega)}}
\def\tildeLdeux{{L^2_{\scriptscriptstyle\!\int\!=0} (\Omega)}}
\def\Hdeuxd{{{   \bf H}^2   (\Omega)}}
\newcommand{\Real}{\mathbb R}

\newcommand{\ie}{i.e.,\@\xspace}
\newcommand{\eg}{e.g.\@\xspace}

\newcommand{\etal}{et al.\@\xspace}

\renewcommand{\ae}{a.e.\@\xspace}

\newcommand{\bH}{{\bf H}}
\newcommand{\bL}{{\bf L}}
\newcommand{\bX}{{\bf X}}

\newcommand{\calC}{{\mathcal C}}
\newcommand{\calF}{{\mathcal F}}
\newcommand{\calI}{{\mathcal I}}
\newcommand{\calN}{{\mathcal N}}
\newcommand{\calT}{{\mathcal T}}

\begin{document}
\title[LBB]{A Note on the Lady\v zenskaja-Babu\v ska-Brezzi Condition}

\author[J.~Guzm{\'a}n]{Johnny Guzm{\'a}n}
\thanks{J.~Guzm{\'a}n is supported by NSF grant DMS-0914596.}
\address[J.~Guzm{\'a}n]{Division of Applied Mathematics\\
Brown University\\
 Providence, RI 02912}
\email{johnny\_guzman@brown.edu}

\author[A.J.~Salgado]{Abner J.~Salgado}
\thanks{AJS is supported by NSF grants CBET-0754983 and DMS-0807811 and an AMS-Simons Travel Grant.}
\address[A.J.~Salgado]{Department of Mathematics\\
University of Maryland\\
College Park, MD 20742}
\email{abnersg@math.umd.edu}

\author[F.-J.~Sayas]{Francisco-Javier Sayas}
\address[F.-J.~Sayas]{Department of Mathematical Sciences\\
University of Delaware\\
 Newark, DE 19716}
\email{fjsayas@math.udel.edu}



\date{Submitted to Journal of Scientific Computing on March 8, 2012.}

\begin{abstract}
The analysis of finite-element-like Galerkin discretization techniques for the stationary
Stokes problem relies on the so-called LBB condition.
In this work we discuss equivalent formulations of the LBB condition.
\end{abstract}

\maketitle

\section{Introduction}
\label{sec:Intro}
The well known Lady\v zenskaja-Babu\v ska-Brezzi (LBB) condition is a particular instance of the
so-called
discrete inf--sup condition which is necessary and sufficient for the well-posedness of discrete
saddle point problems
arising from discretization via Galerkin methods. If $\bX_h$ denotes the discrete velocity space and
$M_h$ the discrete
pressure space, then the LBB condition for the Stokes problem states that there is a constant
$c$ independent of the discretization parameter $h$ such that
\begin{equation}
  c \| q_h \|_{L^2} \leq \sup_{v_h \in \bX_h} 
  \frac{ \int_\Omega (\DIV v_h)\, q_h}{\|v_h\|_{\bH^1}},
  \quad \forall q_h \in M_h.
\tag{LBB}
\label{eq:LBB}
\end{equation}
The reader is referred to \cite{GR86} for the basic theory on saddle point problems on Banach spaces
and 
their numerical analysis. Simply put, this condition sets a structural restriction on the discrete
spaces so that
the continuous level property that the divergence operator is closed and surjective, see
\cite{MR82m:26014,MR1880723}, is preserved uniformly with respect to the discretization parameter.

In the literature the following condition, which we shall denote the generalized LBB condition, is
also assumed
\begin{equation}
c \| \GRAD q_h \|_{\bL^2} \leq
  \sup_{v_h \in \bX_h}
  \frac{ \int_\Omega (\DIV v_h)\, q_h }{ \| v_h \|_{\bL^2} },
  \quad \forall q_h \in M_h,
\tag{GLBB}
\label{eq:wLBB}
\end{equation}
here and throughout we assume $M_h \subset H^1(\Omega)$. By properly defining a discrete gradient
operator, the case of discontinuous pressure spaces can be analyzed with similar arguments to those
that we shall present. Condition \eqref{eq:wLBB}, for example, was used by Guermond
(\cite{MR2210084,MR2334774}) to show that approximate solutions to the three-dimensional Navier
Stokes equations constructed using the Faedo-Galerkin method
converge to a suitable, in the sense of Scheffer, weak solution.
On the basis of \eqref{eq:wLBB}, the same author has also built (\cite{MR2520170}) an
$\bH^s$-approximation theory for the Stokes problem, $0\leq s \leq1$.
Olshanski{\u\i}, in \cite{MR2833487}, under the assumption that the spaces satisfy \eqref{eq:wLBB}
carries out a multigrid analysis for the Stokes problem. Finally, Mardal \etal, 
\cite{schoberlwinther}, use a weighted inf--sup condition to analyze preconditioning techniques
for singularly perturbed Stokes problems (see Section \ref{sec:section5} below).

It is not difficult to show that, on quasi-uniform meshes, \eqref{eq:wLBB} implies \eqref{eq:LBB}, see
\cite{MR2210084}. We include the proof of this result below for completeness.
The question that naturally arises is whether the converse holds. Recall that a well-known result of
Fortin \cite{BF91}
shows that the inf--sup condition \eqref{eq:LBB} is equivalent to the existence of a so-called Fortin
projection that is stable in $\Hunzd$.  In this work, under the assumption that the mesh is 
shape regular and quasi-uniform, we will show that \eqref{eq:wLBB} is equivalent
to the existence of a Fortin projection that has $\bL^2$-approximation properties. Moreover,
when the domain is such that the solution to the Stokes problem 
possesses $\bH^2$-regularity, we will prove that 
\eqref{eq:wLBB} is in fact equivalent to \eqref{eq:LBB}, again on quasi-uniform meshes.

The work by Girault and Scott (\cite{MR1961943}) must be mentioned 
when dealing with the construction of Fortin projection operators with $\bL^2$-approximation properties.
They have constructed such operators for many commonly 
used inf--sup stable spaces, one notable exception being the lowest order Taylor-Hood element 
in three dimensions.
However, \eqref{eq:wLBB} has been shown to hold for the lowest order Taylor-Hood
element directly \cite{MR2210084}.
Our results then can be applied to show that, \eqref{eq:wLBB} is satisfied 
by almost all inf--sup stable finite element spaces, regardless of the smoothness of the domain.

This work is organized as follows.
Section~\ref{sec:prel} introduces the notation and assumptions we shall work with. Condition
\eqref{eq:wLBB} is discussed in Section~\ref{sec:wLBB}. In Section~\ref{sec:Equiv} we actually show the
equivalence of conditions \eqref{eq:LBB} and \eqref{eq:wLBB}, provided the domain is smooth enough.
A weighted inf--sup condition related to uniform preconditioning of the time-dependent Stokes problem
is presented in Section~\ref{sec:section5}, where we show that \eqref{eq:wLBB} implies it.
Some concluding remarks are provided in Section~\ref{sec:conclusion}.

\section{Preliminaries}
\label{sec:prel}
Throughout this work, we will denote by $\Omega \subset \Real^d$ with $d=2$ or $3$ an open bounded
domain with Lipschitz boundary. If additional smoothness of the domain is needed, it will be specified
explicitly.
$\Ldeux$, $\Hun$  and $\Hunz$ denote, respectively, the usual Lebesgue and Sobolev spaces.
We denote by $\tildeLdeux$ the set of functions in $\Ldeux$ with mean zero.
Vector valued spaces will be denoted by bold characters.

We introduce a conforming triangulation $\calT_h$ of $\Omega$ which we assume shape-regular and
quasi-uniform in the sense of \cite{BF91}. The size of the cells in the triangulation is characterized by 
$h>0$. We introduce finite dimensional spaces $\bX_h \subset \Hunzd$ and 
$M_h \subset \tildeLdeux \cap \Hun$ which are
constructed, for instance using finite elements,
on the triangulation $\calT_h$. For these spaces, the inverse inequalities
\begin{equation}
  \| v_h \|_{\bH^1} \leq c h^{-1} \| v_h \|_{\bL^2}, \quad \forall v_h \in \bX_h,
\label{eq:invX}
\end{equation}
and
\begin{equation}
  \| q_h \|_{H^1} \leq c h^{-1} \| q_h \|_{L^2}, \quad \forall q_h \in M_h,
\label{eq:invM}
\end{equation}
hold, see \cite{BF91}. Here and in what follows we denote by $c$ will a constant that is independent of $h$. 

We shall denote by $\calC_h : \Hunzd \rightarrow \bX_h$ the so-called 
Scott-Zhang interpolation operator (\cite{SZ90}) onto the velocity space and we recall that
\begin{equation}
  \| v - \calC_h v \|_{\bL^2} + h \|\calC_h v\|_{\bH^1} \leq c h \| v \|_{\bH^1}, \quad \forall v
\in \Hunzd.
\label{eq:SZprop}
\end{equation}
and
\begin{equation}
  \| v - \calC_h v\|_{\bH^1} \leq c h \| v \|_{\bH^2}, \quad \forall v \in \Hunzd \cap \Hdeuxd
\label{eq:SZh1}
\end{equation}
The Scott-Zhang interpolation operator onto the pressure space
$\calI_h: \tildeLdeux \rightarrow M_h$ can be defined analogously
and satisfies similar stability and approximation properties.
We shall denote by $\pi_h: \Ldeuxd \rightarrow \bX_h$ the $\bL^2$-projection onto $\bX_h$ and
by $\Pi_0 : \Ldeux \rightarrow \Ldeux$ the $L^2$-projection operator onto the space
of piecewise constant functions, \ie
\[
  \Pi_0 q = \sum_{T \in \calT_h} \frac1{|T|}\left(\int_T q\right)\chi_T, \quad \forall q \in \Ldeux.
\]

For one result below we shall require full $\bH^2$-regularity of the solution to the Stokes problem:

\begin{assumption}\label{assumption1}
The domain $\Omega$ is such that for any $f\in \Ldeuxd$, the solution
$(\psi,\theta) \in \Hunzd \times \tildeLdeux$ to the Stokes problem
\begin{equation}
\label{eq:contstokes}
  \begin{dcases}
    -\LAP \psi + \GRAD \theta = f, & \text{in } \Omega, \\
    \DIV \psi = 0, & \text{in } \Omega, \\
    \psi = 0, & \text{on } \partial\Omega,
  \end{dcases}
\end{equation}
satisfies the following estimate:
\begin{equation}
  \| \psi \|_{\bH^2} + \| \theta \|_{H^1} \leq c \| f \|_{\bL^2}.
\label{eq:Cattabriga}
\end{equation}
\end{assumption}

Assumption~\ref{assumption1} is known to hold in two and three dimensions ($d=2,3$) whenever
$\Omega$ is convex or of class $\calC^{1,1}$, see \cite[Theorem 6.3]{MR977489}.

By suitably defining a discrete gradient operator acting on the pressure space, the proofs for 
discontinuous pressure spaces can be carried out with similar arguments.

We introduce the definition of a Fortin projection.
\begin{definition}
\label{def:Fortin}
An operator $\calF_h : \Hunzd \rightarrow \bX_h$ is called a Fortin projection if $\calF_h^2 =
\calF_h$ and
\begin{equation}
  \int_\Omega \DIV(v-\calF_h v)q_h = 0, \quad \forall v \in \Hunzd, \quad \forall q_h \in M_h.
\label{eq:Fortin}
\end{equation}
\end{definition}

We shall be interested in Fortin projections $\calF_h$ that satisfy the condition:
\begin{equation}
  \| \calF_h v \|_{\bH^1} \leq c \| v \|_{\bH^1}, \quad \forall v \in \Hunzd,
\tag{FH1}
\label{eq:fh1}
\end{equation}
or
\begin{equation}
  \| v - \calF_h v \|_{\bL^2} \leq c h \| v\|_{\bH^1}, \quad \forall v \in \Hunzd.
\tag{FL2}
\label{eq:fl2}
\end{equation}

Let us remark that the approximation property \eqref{eq:fl2} implies $\bH^1$-stability.
\begin{lem}
\label{lem:fl2implfh1}
If an operator $\calF_h : \Hunzd \rightarrow \bX_h$ satisfies \eqref{eq:fl2} then it is $\bH^1$-stable, i.e., \eqref{eq:fh1} is satisfied.
\end{lem}
\begin{proof}
The proof relies on the stability and approximation properties \eqref{eq:SZprop} of the Scott-Zhang operator
and on the inverse estimate \eqref{eq:invX}, for if $v \in \Hunzd$,
\begin{align*}
  \| \calF_h v \|_{\bH^1} &\leq \| \calF_h v - \calC_h v \|_{\bH^1} + c \| v \|_{\bH^1}
  \leq c h^{-1} \| \calF_h v - \calC_h v \|_{\bL^2} + c \| v \|_{\bH^1} \\
  &\leq ch^{-1} \| v - \calF_h v \|_{\bL^2} + ch^{-1}\| v - \calC_h v \|_{\bL^2} + c \|v\|_{\bH^1}.
\end{align*}
Conclude using the $\bL^2$-approximation properties of the operators $\calF_h$ and $\calC_h$.
\end{proof}

\begin{rem}
Girault and Scott, \cite{MR1961943}, explicitly constructed a Fortin projection 
that satisfies \eqref{eq:fh1} and \eqref{eq:fl2} for many
commonly used spaces. In fact, they showed that the approximation is local, \ie
\[
  \| \calF_h v - v \|_{\bL^2(T)}+ h_T\| \calF_h v -v\|_{\bH^1(T)}  \leq c h_T \| v
\|_{\bH^1(\calN(T))}, \quad \forall v \in \Hunzd
   \text{ and } \forall T \in \calT_h,
\]
where $\calN(T)$ is a patch containing $T$.
In particular, they have shown the existence of this projection for the Taylor-Hood elements in two 
dimensions. In three dimensions they proved this result for all the Taylor-Hood elements except the 
lowest order case. 
\end{rem}

In this work we shall prove the implications
\[
  \xymatrix{
    \eqref{eq:LBB} \ar@{<=>}[r] \ar@{<=}[d]
    &\exists \calF_h \mbox{\, s.t.\,} \eqref{eq:Fortin} \text{ and } \eqref{eq:fh1} & \\
    \eqref{eq:wLBB} \ar@{<=>}[r]
    &\exists \calF_h \mbox{\, s.t.\,} \eqref{eq:Fortin} \text{ and } \eqref{eq:fl2} &    \eqref{eq:LBB} \text{ and 
Assumption  } \ref{assumption1} \ar@{=>}[l] 
  }
\]
thus showing that, in our setting, all these conditions are indeed equivalent. 
The top equivalence is
well-known, see \cite{BF91,GR86,MR2050138}. The left implication is also known (see
\cite{MR2210084}), for completeness we show
this in Theorem~\ref{thm:wlbbimpllbb}. The bottom implications, although simple to prove, seem to be
new. 

\section{The Generalized LBB Condition}
\label{sec:wLBB}
Let us begin by noticing that the generalized LBB condition \eqref{eq:wLBB} is actually a statement
about coercivity of the $\bL^2$-projection on gradients of functions in the pressure space. Namely,
\eqref{eq:wLBB} is equivalent to
\begin{equation}\label{GLBBb}
  \| \pi_h \GRAD q_h \|_{\bL^2} \geq c \| \GRAD q_h \|_{\bL^2}, \quad \forall q_h \in M_h.
\end{equation}

It is well known that \eqref{eq:wLBB} implies \eqref{eq:LBB}. For completeness we present the proof.
We
begin with a perturbation result.

\begin{lem}
\label{lem:verfurth}
There exists a constant $c$ independent of $h$ such that, for all $q_h \in M_h$, the following holds:
\[
  c \| q_h \|_{L^2} \leq 
  \sup_{v_h \in \bX_h} \frac{ \int_\Omega (\DIV v_h)\, q_h }{\| \GRAD v_h \|_{\bL^2} }
  + h \| \GRAD q_h \|_{\bL^2}.
\]
\end{lem}
\begin{proof}
The proof relies on the properties \eqref{eq:SZprop} of the Scott-Zhang interpolation operator $\calC_h$,
\begin{align*}
  c \|q_h\|_{L^2} &\leq \sup_{v \in \Hunzd} \frac{ \int_\Omega (\DIV v)\, \, q_h }{\| \GRAD v
\|_{\bL^2} }
  \leq 
  \sup_{ v \in \Hunzd} \frac{ \int_\Omega (\DIV\,\calC_h v)\, q_h }{\| \GRAD (\calC_h v)
\|_{\bL^2} }
  +
  \sup_{ v \in \Hunzd}\frac{\int_\Omega\big(\DIV\left(v - \calC_h v \right)\big)q_h }{\|\GRAD
v\|_{\bL^2} }
  \\ &\leq
  \sup_{ v_h \in \bX_h} \frac{ \int_\Omega (\DIV v_h) \, q_h }{\| \GRAD v_h \|_{\bL^2} }
  + \sup_{ v \in \Hunzd} \frac{ \int_\Omega\left(v - \calC_h v \right)\SCAL\GRAD q_h}{\|\GRAD
v\|_{\bL^2}},
\end{align*}
conclude using \eqref{eq:SZprop}.
\end{proof}
On the basis of Lemma~\ref{lem:verfurth} we can readily show that \eqref{eq:wLBB} implies \eqref{eq:LBB}. Again,
this result is not new and we only include the proof for completeness.
\begin{thm}
\label{thm:wlbbimpllbb}
\eqref{eq:wLBB} implies \eqref{eq:LBB}.
\end{thm}
\begin{proof}
Since we assumed that $M_h \subset \tildeLdeux \cap \Hun$, the proof is straightforward:
\[
  \sup_{ v_h \in \bX_h} \frac{ \int_\Omega (\DIV v_h)\, q_h }{\|\GRAD v_h \|_{\bL^2} }
  =
  \sup_{ v_h \in \bX_h} \frac{ \int_\Omega v_h \SCAL \GRAD q_h }{\|\GRAD v_h \|_{\bL^2} }
  \geq
  \frac{ \int_\Omega \pi_h \GRAD q_h \SCAL\GRAD q_h }{\|\GRAD \pi_h \GRAD q_h \|_{\bL^2} }
  =
  \frac{ \|\pi_h \GRAD q_h \|_{\bL^2}^2 }{ \| \GRAD \pi_h \GRAD q_h \|_{\bL^2} }
  \geq
  c h \|\pi_h \GRAD q_h \|_{\bL^2}
\]
where, in the last step, we used the inverse inequality \eqref{eq:invX}.
This, in conjunction with Lemma~\ref{lem:verfurth} and the characterization \eqref{GLBBb}, implies the result.
\end{proof}

Let us now show that the generalized LBB condition 
\eqref{eq:wLBB} is equivalent to the existence of a Fortin operator satisfying \eqref{eq:fl2}.
We begin with a modification of a classical result.

\begin{lem}
\label{lem:tartarwithh}
For all $p\in \Hun$ there is $v\in\Hunzd$ such that
\[
    \DIV v = p - \Pi_0 p, \qquad v|_{\partial T}=0 \quad \forall T\in \calT_h,
\]
and
\[
  \| v \|_{\bL^2} \leq c \left( \sum_{T \in \calT_h} h_T^4 \| \GRAD p \|_{\bL^2(T)}^2 \right)^{1/2}.
\]
\end{lem}
\begin{proof}
Let $p\in \Hun$ and $T \in \calT_h$. Clearly,
\[
  \int_T p - \Pi_0 p = 0.
\]
A classical result (\cite{MR82m:26014,MR1846644,GR86,MR1880723}) implies that there is a 
$v_T \in \bH^1_0(T)$ with
$ \DIV v_T = p - \Pi_0 p$
in $T$ and
\begin{equation}
  \| \GRAD v_T \|_{\bL^2(T)} \leq c \| p - \Pi_0 p \|_{L^2(T)}.
\label{eq:saves}
\end{equation}
Given that the mesh is assumed to be shape regular, by mapping to the reference element it is seen
that
the constant in the last inequality does not depend on $T \in \calT_h$.

Let $v \in \Hunzd$ be defined as $ v|_T = v_T$ for all $T$ in $\calT_h$. By construction,
\[
  \DIV v = p - \Pi_0 p, \quad \text{\ae in } \Omega.
\]
Moreover,
\[
  \| v \|_{\bL^2}^2 = \sum_{T \in \calT_h} \| v \|_{\bL^2(T)}^2
    \leq c \sum_{T \in \calT_h} h_T^2 \| \GRAD v \|_{\bL^2(T)}^2
    \leq c \sum_{T \in \calT_h} h_T^2 \| p - \Pi_0 p \|_{L^2(T)}^2
    \leq c \sum_{T \in \calT_h} h_T^4 \| \GRAD p \|_{\bL^2(T)}^2.
\]
The first equality is by definition; then we applied the Poincar\'e-Friedrichs inequality (since
$v|_T = v_T \in \bH^1_0(T)$); next we used the properties of the function $v_T$ and the
approximation properties of the projector $\Pi_0$.
\end{proof}

With this result at hand we can prove the following.

\begin{thm}
\label{thm:fl2implwlbb}
If there exists a Fortin operator $\calF_h$ that satisfies \eqref{eq:fl2}, then
\eqref{eq:wLBB} holds.
\end{thm}
\begin{proof}
Let $q_h \in M_h$. Using the properties of the operator $\Pi_0$ and 
the local analogue of
the inverse inequality \eqref{eq:invM}, we get
\[
  \| \GRAD q_h \|_{\bL^2}^2 
  = \sum_{T \in \calT_h} \left\| \GRAD \left(q_h - \Pi_0 q_h \right) \right\|_{\bL^2(T)}^2
  \leq \sum_{T \in \calT_h} \frac1{h_T^2} \| q_h - \Pi_0 q_h \|_{\bL^2(T)}^2
  \leq \frac{c}{h^2} \| q_h - \Pi_0 q_h \|_{\bL^2}^2.
\]
From Lemma~\ref{lem:tartarwithh} we know there exists $v \in \Hunzd$ with $\DIV v = q_h - \Pi_0 q_h$
and
\[
  \| v \|_{\bL^2} \leq c h^2 \| \GRAD q_h \|_{\bL^2},
\]
hence
\[
  \| \GRAD q_h \|_{\bL^2}^2 \leq \frac{c}{h^2} \| q_h - \Pi_0 q_h \|_{L^2}^2
  = \frac{c}{h^2} \int_\Omega (\DIV v) \, (q_h-\Pi_0 q_h)
  = \frac{c}{h^2} \int_\Omega (\DIV v) \, q_h,
\]
where the last inequality follows from integration by parts over each $T$ and using the fact that $v|_{\partial T} = 0$ (see Lemma~\ref{lem:tartarwithh}).

Using the existence of the operator $\calF_h$,
\[
  \| \GRAD q_h \|_{\bL^2}^2 \le \frac{c}{h^2} \int_\Omega (\DIV \,\calF_h v) q_h
  \leq \left( \sup_{ w_h \in \bX_h} \frac{ \int_\Omega (\DIV w_h) \, q_h }{ \| w_h \|_{\bL^2}
}\right) 
  \frac{c}{h^2}\| \calF_h v \|_{\bL^2}.
\]
It remains to show that 
\[
  \| \calF_h v \|_{\bL^2} \leq c h^2 \| \GRAD q_h \|_{\bL^2}.
\]
For this purpose, we use the approximation property \eqref{eq:fl2} and
Lemma~\ref{lem:tartarwithh}
\[
  \| \calF_h v \|_{\bL^2} \leq \| \calF_h v - v\|_{\bL^2} + \|v \|_{\bL^2}
  \leq ch \| \GRAD v \|_{\bL^2} + c h^2 \|\GRAD q_h \|_{\bL^2}
  \leq ch^2 \| \GRAD q_h \|_{\bL^2},
\]
where the last inequality holds because of \eqref{eq:saves}.
\end{proof}

The converse of Theorem~\ref{thm:fl2implwlbb} is given in the following.

\begin{thm}
\label{thm:wlbbimplfl2}
If \eqref{eq:wLBB} holds, then there exists a Fortin projector $\calF_h$ that satisfies
\eqref{eq:fl2}.
\end{thm}
\begin{proof}
Let $v\in\Hunzd$. Define $(z_h, p_h) \in \bX_h \times M_h$ as the solution of
\begin{equation}
  \begin{dcases}
    \int_\Omega z_h\SCAL w_h -\int_\Omega p_h \DIV w_h = \int_\Omega v\SCAL w_h, & \forall w_h \in \bX_h, \\
    \int_\Omega q_h \DIV z_h  = \int_\Omega q_h \DIV v, & \forall q_h \in M_h.
  \end{dcases}
\label{eq:l2stokes}
\end{equation}
Notice that \eqref{eq:wLBB} provides precisely necessary and sufficient conditions for this problem
to have a
unique solution.

Define $\calF_h v := z_h$ we claim that this is indeed a Fortin projection that satisfies
\eqref{eq:fl2}.
By construction, \eqref{eq:Fortin} holds (see the second equation in \eqref{eq:l2stokes}). To show
that this is indeed a
projection, assume that $v=v_h \in \bX_h$ in \eqref{eq:l2stokes}, setting $w_h = z_h - v_h$ we
readily obtain that
\[
  \| z_h - v_h \|_{\bL^2}^2 =0.
\]
It remains to show the approximation properties of this operator. We begin by noticing that
\eqref{eq:wLBB} implies
\begin{equation}
  c \| \GRAD p_h \|_{\bL^2} \leq  \sup_{w_h \in \bX_h} \frac{ \int_\Omega p_h \DIV w_h }{ \| w_h
\|_{\bL^2}}
    \leq \sup_{w_h \in \bX_h} \frac{ \int_\Omega (v-\calF_h v)\SCAL w_h  }{ \| w_h \|_{\bL^2}}
    \leq \| v - \calF_h v \|_{\bL^2},
\label{eq:boundgradp}
\end{equation}
where we used \eqref{eq:l2stokes}. To obtain the approximation property \eqref{eq:fl2} we use the
Scott-Zhang interpolation operator $\calC_h$,
\begin{align*}
  \| \calF_h v - v \|_{\bL^2}^2 &= \int_\Omega ( \calC_h v - v )\SCAL( \calF_h v - v )
  + \int_\Omega ( \calF_h v - \calC_h v )\SCAL( \calF_h v - v ) \\
  & \leq
  \| \calC_h v - v \|_{\bL^2}\| \calF_h v - v \|_{\bL^2}
  +  \int_\Omega ( \calF_h v - \calC_h v )\SCAL( \calF_h v - v ).
\end{align*}
We bound the first term using the approximation property \eqref{eq:SZprop} of $\calC_h$. To bound
the second
term we use problem \eqref{eq:l2stokes} with $w_h = \calF_h v - \calC_h v$, then
\[
  \int_\Omega ( \calF_h v - \calC_h v )\SCAL( \calF_h v - v )
  = \int_\Omega p_h \DIV( \calF_h v - \calC_h v )
  = \int_\Omega p_h \DIV( v - \calC_h v )
  = -\int_\Omega \GRAD p_h \SCAL ( v - \calC_h v ),
\]
we conclude applying the Cauchy-Schwarz inequality and using \eqref{eq:boundgradp}.
\end{proof}

\section{Smooth Domains}
\label{sec:Equiv}
Here we show that, provided \eqref{eq:LBB} holds and, moreover,
the domain $\Omega$ is such that Assumption~\ref{assumption1} is satisfied, then \eqref{eq:fl2}
holds and hence \eqref{eq:wLBB} holds as well. This is shown in the following.

\begin{thm}
\label{thm:lbbimplwlbb}
Assume the domain $\Omega$ is such that the solution to
\eqref{eq:contstokes} possesses $\bH^2$-elliptic regularity, \ie Assumption \ref{assumption1}
holds. Then \eqref{eq:LBB} implies that there is a Fortin operator
$\calF_h$ that satisfies \eqref{eq:fl2}.
\end{thm}
\begin{proof}
Let $v\in \Hunzd$. Define $(z_h,p_h) \in \bX_h \times M_h$ as the solution to the discrete Stokes
problem
\begin{equation}
  \begin{dcases}
    \int_\Omega \GRAD z_h \SSCAL \GRAD w_h -\int_\Omega p_h \DIV w_h = \int_\Omega \GRAD v \SSCAL \GRAD w_h, & \forall w_h
\in \bX_h, \\
    \int_\Omega q_h \DIV z_h = \int_\Omega q_h \DIV v, & \forall q_h \in M_h,
  \end{dcases}
\label{eq:stokes}
\end{equation}
where, in \eqref{eq:stokes}, the colon is used to denote the tensor product of matrices. Notice that \eqref{eq:LBB} implies that
this problem always has a unique solution.

Set $\calF_h v := z_h$. Proceeding as in the proof of Theorem~\ref{thm:wlbbimplfl2} we see that this is 
indeed a projection. Moreover, \eqref{eq:Fortin} holds by construction.
It remains to
show that \eqref{eq:fl2} is satisfied. To this end, analogously to the proof of
Theorem~\ref{thm:wlbbimplfl2},
we notice that \eqref{eq:LBB} implies
\[
  \| p_h \|_{L^2} \leq c \| \GRAD (\calF_h v - v )\|_{\bL^2}.
\]
We now argue by duality. Let $\psi$ and $\phi$ solve
\eqref{eq:contstokes} with $f=\calF_h v- v$. Assumption \eqref{eq:Cattabriga} then implies
\begin{align*}
  \| \calF_h v - v \|_{\bL^2}^2 &= \int_\Omega (\calF_h v - v)\SCAL(-\LAP\psi + \GRAD \theta) \\
  &= \int_\Omega \GRAD(\calF_h v -v ):\GRAD(\psi-\calC_h \psi) -\int_\Omega(\theta - \calI_h
\theta)\,\DIV(\calF_h v -v )\, \\
  & \phantom{=}+ \int_\Omega \GRAD(\calF_h v -v ):\GRAD(\calC_h \psi) - \int_\Omega (\calI_h \theta)\,\DIV(\calF_h v -v )\,
\end{align*}
Notice that since $\calI_h \theta\in M_h$, $\int_\Omega (\calI_h \theta)\,\DIV(\calF_h v -v )=0$.
Since $\DIV \psi = 0$,
using \eqref{eq:stokes}, the estimate for $p_h$, \eqref{eq:SZh1} and \eqref{eq:Cattabriga},
\[
  \int_\Omega \GRAD(\calF_h v -v ):\GRAD(\calC_h \psi) = \int_\Omega p_h \DIV (\calC_h \psi-\psi)
  \leq ch \| v - \calF_h v \|_{\bH^1} \| v - \calF_h v \|_{\bL^2}.
\]
A direct application of of \eqref{eq:SZh1}, \eqref{eq:SZprop} and \eqref{eq:Cattabriga} allows us to
obtain the following estimates:
\[
  \int_\Omega (\theta - \calI_h \theta)\,\DIV(\calF_h v -v ) 
  + \int_\Omega \GRAD(\calF_h v -v )\SSCAL\GRAD(\psi-\calC_h \psi)
  \leq
   c h \| \calF_h v - v \|_{\bL^2}\| v \|_{\bH^1}
\]

We conclude using a stability estimate for \eqref{eq:stokes}
\[
  \| \calF_h v - v \|_{\bL^2} \leq c h \| \calF_h v - v \|_{\bH^1} \leq c h \| v \|_{\bH^1},
\]
which, given \eqref{eq:LBB}, is uniform in $h$.
\end{proof}

\section{The Weighted LBB condition}
\label{sec:section5}

In relation to the construction of uniform preconditioners for
discretizations of the time dependent Stokes problem, Mardal, Sch\"oberl and Winther, \cite{schoberlwinther}, 
consider the following inf--sup condition,
\begin{equation}
  c\| q_h \|_{H^1+\epsilon^{-1} L^2} \leq 
  \sup_{v_h \in \bX_h} \frac{ \int_\Omega \DIV v_h q_h}{\|v_h\|_{\bL^2 \cap \epsilon \bH^1}}, 
  \quad \forall q_h \in M_h.
\label{eq:winf-sup}
\end{equation}
where
\[
\| q \|_{H^1+\epsilon^{-1} L^2}^2= \inf_{q_1+q_2=q} 
  \left( \|q_1\|_{H^1}^2+ \epsilon^{-2} \|q_2\|_{L^2}^2 \right),
\]
and
\[
  \|v\|_{\bL^2 \cap \epsilon \bH^1}^2 = \|v\|_{\bL^2}^2 + \epsilon^2 \|v\|_{\bH^1}^2.
\]

By constructing a Fortin projection operator that is $\bL^2$-bounded they have showed, on quasi-uniform meshes, 
that the inf--sup condition \eqref{eq:winf-sup} holds for the lowest order Taylor-Hood element in two dimension.
In addition, they proved the same result, on shape regular meshes, for the mini-element.
Here, we show that \eqref{eq:winf-sup} holds if we assume \eqref{eq:wLBB}. A simple consequence of this 
result is that, on quasi-uniform meshes, \eqref{eq:winf-sup} holds for any order Taylor-Hood elements
in two and three dimensions.

\begin{thm}
\label{thm:wlbbimplweightlbb}
Let $\Omega$ be star shaped with respect to ball.
If the spaces $\bX_h$ and $M_h$ are such that \eqref{eq:wLBB} is satisfied,
then the inf--sup condition \eqref{eq:winf-sup} holds with a constant that 
does not depend on $\epsilon$ or $h$.
\end{thm}
\begin{proof}
We consider two cases: $\epsilon \ge h$ and $\epsilon < h$.

Given that the domain $\Omega$ is star shaped with respect to a ball, 
we can conclude (\cite{schoberlwinther}) that the following 
continuous inf--sup condition holds,
\begin{equation}\label{conweight}
c \| q \|_{H^1+\epsilon^{-1} L^2} \leq \sup_{v \in \Hunzd}
  \frac{ \int_\Omega q\, \DIV v }{ \|v\|_{\bL^2 \cap \epsilon \bH^1}},
  \quad \forall q \in \tildeLdeux,
\end{equation}
with a constant $c$ independent of $\epsilon$.

We first assume that $\epsilon \ge h$. Using \eqref{conweight} for $q_h \in M_h$ we have, 
\begin{align*}
  c \| q_h \|_{H^1+\epsilon^{-1} L^2} &\leq
  \sup_{v \in \Hunzd } \frac{ \int q_h\,\DIV v }{ \| v \|_{\bL^2 \cap \epsilon \bH^1 } } =
  \sup_{v \in \Hunzd } \frac{ \int_\Omega q_h\,\DIV (\calF_h v)  }{\|\calF_h v\|_{\bL^2 \cap \epsilon \bH^1}}
    \frac{\|\calF_h v\|_{\bL^2 \cap \epsilon \bH^1}}{\|v\|_{\bL^2\cap \epsilon \bH^1} } \\
  &\leq \sup_{v_h \in \bX_h} \frac{ \int_\Omega q_h\,\DIV v_h  }
    {\|v_h\|_{\bL^2 \cap \epsilon \bH^1} }
  \sup_{ v \in \Hunzd} \frac{\|\calF_h v\|_{\bL^2 \cap \epsilon \bH^1}}{\|v\|_{\bL^2 \cap \epsilon \bH^1}},
\end{align*}
where we used that, since \eqref{eq:wLBB} holds, Theorem~\ref{thm:wlbbimplfl2} shows that 
there exists a Fortin operator $\calF_h$ that satisfies
\eqref{eq:Fortin}. By Lemma~\ref{lem:fl2implfh1} and
the approximation properties \eqref{eq:fl2} of the Fortin operator,
\begin{align*}
  \|\calF_h v \|_{\bL^2 \cap \epsilon \bH^1} &\leq
  c \left( \|\calF_h v \|_{\bL^2} + \epsilon \| \calF_h v \|_{\bH^1}   \right) \leq
  c \left( \| v \|_{\bL^2} + \|v - \calF_h v \|_{\bL^2} + \epsilon \| v \|_{\bH^1}   \right) \\
  &\leq c \left(  \|v\|_{\bL^2} + (\epsilon + h)\| v \|_{\bH^1} \right)
  \leq c \left(  \|v\|_{\bL^2} + 2\epsilon\| v \|_{\bH^1} \right)
  \leq c\| v \|_{\bL^2 \cap \epsilon \bH^1},
\end{align*}
where we used that $h \leq \epsilon$.

On the other hand, if $\epsilon < h$ we use $q_1 = q_h$ and $q_2 = 0$ in the definition of the weighted norm for the pressure space.
Condition \eqref{eq:wLBB} then implies
\[
  \| q_h \|_{H^1+\epsilon^{-1} L^2} \leq
 c \| \GRAD q_h \|_{\bL^2} \leq c \sup_{v_h \in \bX_h} \frac{ \int_\Omega q_h\,\DIV v_h  }{ \| v_h \|_{\bL^2} }
 \leq c \sup_{v_h \in \bX_h} \frac{ \int_\Omega q_h\, \DIV v_h  }{ \| v_h \|_{\bL^2 \cap \epsilon \bH^1} }
 \sup_{v_h \in \bX_h} \frac{ \| v_h \|_{\bL^2 \cap \epsilon \bH^1} }{ \| v_h \|_{\bL^2} }.
\]
By the inverse inequality \eqref{eq:invX},
\[
  \frac{ \| v_h \|_{\bL^2 \cap \epsilon \bH^1} }{ \| v_h \|_{\bL^2} } \leq
  c\left( 1 + \epsilon h^{-1} \right).
\]
Conclude using that $\epsilon < h $.
\end{proof}

\section{Concluding Remarks}
\label{sec:conclusion}
There seems to be one main drawback to our methods of proof. Namely, all our results rely heavily on
the fact that we have a quasi-uniform mesh. However, at the present moment we do not know whether this
condition can be removed. Finally, it will be interesting to see if \eqref{eq:LBB} is in fact equivalent to
\eqref{eq:wLBB} on domains that do not satisfy the regularity assumption 
\eqref{eq:Cattabriga} (\eg non convex polyhedral domains).

On the other hand, it seems to us that condition \eqref{eq:wLBB} must be regarded as the most important one.
Our results show that, under the sole assumption that the mesh is quasi-uniform, this condition implies
the classical condition \eqref{eq:LBB} (Theorem~\ref{thm:wlbbimpllbb}). Moreover, as shown in 
Theorem~\ref{thm:wlbbimplweightlbb}, this condition implies the weighted inf--sup condition \eqref{eq:winf-sup} on quasi-uniform meshes.

\bibliographystyle{plain}
\bibliography{biblio}

\end{document}